\documentclass[11pt]{amsart}
\usepackage{amsmath}
\usepackage{amsfonts}
\usepackage{amsthm}
\usepackage{amssymb}
\usepackage{mathrsfs}
\usepackage{latexsym}
\usepackage{dsfont}
\usepackage{hyperref}
\usepackage{graphicx}
\usepackage{float}

\def\tareesidedbox#1{\setbox0=\hbox{$#1$}\dimen0=\wd0 \advance\dimen0 by3pt\rlap{\hbox{\vrule height9pt width.4pt
 depth2pt \kern-.4pt\vrule height9.4pt width\dimen0 depth-9pt\kern-.4pt \vrule height9pt width.4pt depth2pt}}
 \relax \hbox to\dimen0{\hss$#1$\hss}}
\def\ho#1{\tareesidedbox#1}

\swapnumbers

\newtheorem{theorem}{Theorem}[section]

\newtheorem{df}[theorem]{Definition}
\newtheorem{lemma}[theorem]{Lemma}
\newtheorem{prop}[theorem]{Proposition}
\newtheorem{corr}[theorem]{Corollary}
\newtheorem{remark}[theorem]{Remark}

\def\Ch{\mathrm{Ch}}
\def\Z{\mathbf{Z}}
\def\Q{\mathbf{Q}}
\def\R{\mathbf{R}}
\def\eps{\epsilon}
\def\veps{\varepsilon}
\def\OL{\mathcal{O}}

\def\Tr{\mathrm{Tr}}
\def\M{\mathscr{M}}
\def\N{\mathscr{N}}
\def\k{\underline{k}}
\def\keps{\underline{\eps}}
\begin{document}

\title{Abelian Spiders}
\author{Frank Calegari and Zoey Guo}
\thanks{The authors were supported in part by NSF  Grant
  DMS-1404620.}
 \maketitle
  
  \section{Introduction}
  
  Let~$\Gamma$ be a connected finite graph. Fix an integer~$k$, and let $v_1, \ldots, v_k$ be~$k$
  (not necessarily distinct)  vertices of~$\Gamma$. For any~$k$-tuple $\k = (r_1,\ldots,r_k)$ of non-negative integers, we define
  a ($k$-)\emph{spider graph} $\Gamma_{\k}$ on~$\Gamma$ to be the graph  obtained by adjoining a~$2$-valent tree
  of length~$r_i$ to~$\Gamma$ at~$v_i$. We say a graph~$\Gamma$ is~\emph{abelian} if~$\Q(\lambda^2)$
  is an abelian extension, where~$\lambda$ is the Perron--Frobenius eigenvalue of~$\Gamma$ (the
unique  largest real eigenvalue of the adjacency matrix~$M_{\Gamma}$ of~$\Gamma$).
 If~$\Gamma$ is one of the affine Dynkin diagrams, then~$\Gamma$ is abelian, and~$\lambda^2 = 4 \cos^2(2 \pi/N)$ for some integer~$N$. Conversely, if~$\lambda \le 2$, then~$\Gamma$ is an affine Dynkin diagram.

  \begin{theorem} \label{theorem:one} Fix~$\Gamma$ and~$k$. There are only
  finitely many abelian~$k$-spiders~$\Gamma_{\k}$ which are not  Dynkin
  diagrams.
There is an effective algorithm for determining all such spiders.
   \end{theorem}
  
  \begin{remark}  \emph{If~$\Gamma$ is not  already of the form~$A_n$ or~$D_n$,
  then only finitely many of the spiders~$\Gamma_{\k}$ will be Dynkin diagrams.}
    \end{remark}
  
  One motivation for this paper is the application to subfactors, as in~\cite{CMS}. 
  One of the main results (Theorem~1.0.3) of~\cite{CMS} was a version of Theorem~\ref{theorem:one}
  for~$1$-spiders.  The paper~\cite{CMS} also contained a weaker result (Theorem~1.0.6) which
  was sufficient for the application
  to subfactors but had the advantage that the effective constants could be made explicit.
  In contrast, Theorem~\ref{theorem:one} already comes
  with  computable effective constants, and moreover these constants will be small enough
  that our results are ``effectively effective'' in many cases
   (although there is certainly
   some combinatorial explosion as~$k$ increases). In order for this to be so,
  we have worked hard in this paper to make our results as tight as possible, even when weaker estimates
 would certainly suffice to prove the main theorem. As
 an application of Theorem~\ref{theorem:one} to the theory
 of subfactors, we
 prove the following result, conjectured by S.~Morrison~\cite{MorrisonNotes}. Let $\Gamma_{a,b}$ denote the ``Morrison spider,'' given as follows:
\begin{figure}[!ht] \label{fig:one}
\begin{center}
  \includegraphics[width=60mm]{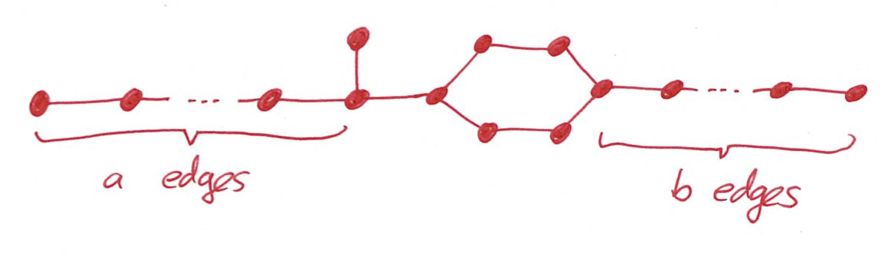}
  \end{center}
\end{figure}

\begin{theorem} \label{theorem:morrison} The spider~$\Gamma_{a,b}$ above is abelian only when $(a,b) = (0,0)$ or $(1,1)$.
\end{theorem}

\section{Acknowledgments}
We thank Scott Morrison for bringing to our attention the problem of understanding
the Perron--Frobenius eigenvalues arising in Theorem~\ref{theorem:morrison}.
We also thank Chris Smyth for bringing to our attention the papers~\cite{Zone,Ztwo}, where
calculations similar to Prop.~\ref{prop:list} are carried out.

\section{Estimates} 
The first technical ingredient is the following inequality below, which is similar (but quite
a bit more complicated) to  Lemma~4.2.3 of~\cite{CMS}. In fact, it turns out that inequalities
of a similar shape were first considered by Smyth in~1981~\cite{Smythtwo,Smyth}, where the intended
application was to generalizations of Siegel's theorem on lower bounds for the trace
of totally positive integers. The creation of such inequalities seems to be part science and part art.
 Let $\Ch_{N}(x)$ denote the minimal polynomial of $(\zeta_N + \zeta^{-1}_N)^{2}$. The table below contains
 explicit expressions for the~$\Ch_{N}(x)$ together with the value of~$\M(\zeta_N + \zeta^{-1}_N)$, where~$\M(\beta):= \frac{\Tr_{K/\Q}(\beta^2)}{[K:\Q]} $ is the normalized trace of~$\beta^2$. The coefficient~$a_N$ is used below in the definition of~$B(x)$.
 The optimization of the coefficients~$a_N$ in the definition of~$B(x)$ was performed by simulated annealing.

\begin{center} 
 \begin{tabular}{clcc}
 \hline
 $N$ & $\Ch_{N}(x)$ & $\M(\zeta_N + \zeta^{-1}_N)$  & $a_N$ \\
 $1$ & $x - 4$ & $4$ & $673$ \\
 $3$ & $x - 1$ & $1$ & $6$ \\
 $4$ & $x$ & $0$ & $4$  \\
 $5$ & $x^2 - 3x + 1$ & $3/2$ & $2$  \\
 $7$ & $x^3 - 5 x^2 + 6 x - 1$ & $5/3$  & $5$ \\
 $8$ & $x - 2$ & $2$ & $157$ \\
 $9$ & $x^3 - 6 x^2 + 9 x - 1$ & $2$ & $13$ \\
 $12$ & $x - 3$ & $3$ & $578$ \\
 $15$ & $x^4 - 9x^3 + 26x^2 - 24x + 1$ & $9/4$ & $43$  \\
 $16$ & $x^2 - 4 x + 2$ & $2$& $49$  \\
 $20$ & $x^2 - 5 x + 5$ & $5/2$ & $215$  \\
 $21$ & $x^6 - 13x^5 + 64 x^4 - 146 x^3 + 148 x^2 - 48 x + 1$
 & $13/6$  & $10$ \\
 $24$ & $x^2 - 4 x + 1$ & $2$ & $25$  \\
 $28$ & $x^3 - 7 x^2 + 14 x - 7$ & $7/3$  &  $80$ \\
 $44$ & $x^5 - 11 x^4 + 44 x^3 - 77 x^2 + 55 x - 11$ & $11/5$  & $24$ \\
 $52$ & $x^6 - 13 x^5 + 65 x^4 - 156 x^3 + 182 x^2 - 91 x + 13$ & $13/6$  & $1$ \\
 \end{tabular}
\end{center}

If~$N$ is not on this list, set~$a_N = 0$.
  This list of polynomials includes every~$N$ where the inequality~$\M(\zeta_N + \zeta^{-1}_N) > 13/6$ is
  satisfied, as well as a complete list
of all such polynomials for~$N < 11$.

\begin{df}
Define the function~$B(x)$ as follows:
$$B(x) = \frac{9}{4} - x -  \frac{1}{1000} \sum a_N \log | \Ch_N(x)|.$$
 \end{df}
 
The key property of $B(x)$ is the following estimate:
\begin{lemma} \label{lemma:sneak} For $x \in [0,4]$ where $B(x)$ is defined, $B(x) \ge 0$.
\end{lemma}

\begin{figure}[H] \label{fig:B}
\begin{center}
  \includegraphics[width=160mm]{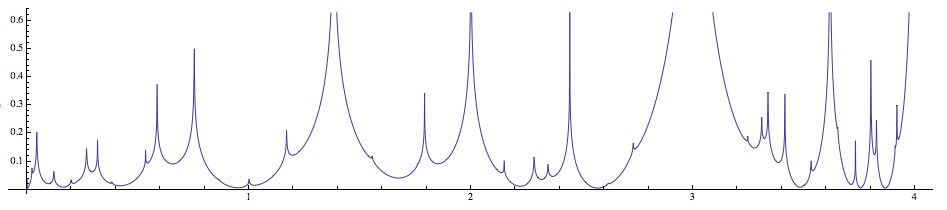}
  \end{center}
  \caption{The graph of $ B(x)$ in $[0,4]$.}
\end{figure}

The derivative of~$B(x)$ lies in $\Q(x)$.
The minimum value of~$B(x)$ in~$[0,4]$ occurs at an algebraic number~$\alpha
\sim 0.00209304$ of degree~$40$, where~$B(x)$ obtains the value~$\sim 0.00599001$.
For~$x > 4$ (away from singularities), $B(x)$ is decreasing. One
 has the estimate $\lim_{x \rightarrow \infty} B(x)/x = -1$.

\begin{theorem} \label{theorem:bounds} Let~$L$ be a non-negative
real number, and let $\beta$ be a totally real algebraic integer with $K = \Q(\beta^2)$ such that:
\begin{enumerate}
\item $\beta^2$ is not a singularity of~$B(x)$,
\item The largest conjugate $\ho{\beta}$ of~$\beta$ satisfies $\ho{\beta} < L$.
\item At most~$M$ conjugates of~$\beta^2$ lie outside the interval~$[0,4]$.
\end{enumerate}
Then,
$\displaystyle{\M(\beta) = \frac{\Tr_{K/\Q}(\beta^2)}{[K:\Q]} < \frac{14}{5}}$
if either $B(L^2) > 0$ or 
$\displaystyle{[K:\Q] \ge \frac{20}{11} \cdot M \cdot |B(L^2)|}$.
\end{theorem}

\begin{proof} At most~$M$ conjugates of~$\beta^2$ lie outside the interval~$[0,4]$.
Consider the sum
$\sum B(\sigma \beta^2)$.
If~$\sigma \beta^2$ is not a singularity of~$B$, then the sum of each logarithmic term
is a negative rational number times the logarithm of the norm of an algebraic integer, and
is hence negative.  If $D = [\Q(\beta^2):\Q]$, it follows that
$$\sum B(\sigma \beta^2) \le \frac{9}{4} \cdot D  - D \M(\beta).$$
On the other hand, we have the estimate $B(x) \ge 0$ for $x \in [0,4]$, and
that~$B(x)$ is decreasing otherwise. Hence
$$\sum B(\sigma \beta^2) \ge M \cdot B(L^2).$$
(Note that $B(L^2) \le 0$ for $L > 2.0152$ or so).
Combining these estimates, we deduce that
$$\frac{9}{4} -  \M(\beta) - \frac{M \cdot B(L^2)}{D} = 
\frac{14}{5} -  \M(\beta)     - \frac{11}{20} - \frac{M \cdot B(L^2)}{D}  \ge 0,$$
which is a contradiction as soon as either of the inequalities of
the statement are satisfied.
\end{proof}

\subsection{The spectrum of~\texorpdfstring{$\Gamma_{\k}$}{Gammak}}
\label{section:spectrum}
We begin by recording some basic properties of eigenvalues of graphs. A reference
for this section is~\cite{Salem}. The following Lemma is essentially Lemma~12 of~\cite{Salem}:

\begin{lemma} If~$r_i \ge 2$ for all~$i$, then the characteristic polynomial~$P_{\k}(x)$ of~$\Gamma_{\k}$
has the form:
$$\left(t - \frac{1}{t}\right)^k  P_{\k}\left(x \right)= \sum_{\keps}
t^{\sum \eps_i r_i} P_{\keps}\left(x \right),$$
where~$x = t + t^{-1}$,  the index $\keps$ runs over~$k$-tuples $(\eps_1,\ldots,\eps_k)$ with $\eps_i \in \{1,-1\}$,
and 
where the polynomials $P_{\keps} \in \Z[x]$ do not depend on~$\k$.\end{lemma}

\medskip

Let~$Q(x) = F_{\eps}(x)$ where $\eps = (1,1,\ldots,1)$.
Let~$S$ denote the set of real roots of~$Q(x)$ in $(2,\infty)$, counted with multiplicity.
Say that a vector~$\k$ is large if \emph{all} the entries~$r_i$ are large.

\begin{lemma}  \label{lemma:general} We have the following: 
\begin{enumerate}
\item If $\k' > \k$ in the partial ordering, then the  Perron Frobenius eigenvalue~$\lambda'$ of
$\Gamma_{\k'}$ is strictly larger than~$\lambda$.
\item $P_{\k}(x)$ has~$|S|$ real roots  $> 2$ for sufficiently large~$\k$, and they are converging from below to~$S$.
\end{enumerate}
\end{lemma}

\begin{proof} The first claim follows from the interlacing Theorem. The second claim is 
 proved in~\cite{Salem}. The main point is that any root $> 1 + \veps$ of $P_{\k}(t + t^{-1})$
 will continue (by interlacing)
to be $> 1 + \veps$ as $\k$ grows. Then, for sufficiently large~$\k$, Rouch\'{e}'s theorem will show that
the number  of real roots~$>1 + \veps$ of $P_{\k}(t + t^{-1})$ will
be equal to the number of real roots of~$Q(t + t^{-1})$. 
\end{proof}

We immediately deduce:

\begin{lemma} \label{lemma:uniform} There exists constants~$M = M_{\Gamma}$ and~$L = L_{\Gamma}$ such that:
\begin{enumerate}
\item If~$\lambda$ is the Perron--Frobenius eigenvalue of~$\Gamma_{\k}$, then
$\lambda^2 - 2 \le L$.
\item At most~$M$ of the conjugates of $\lambda^2 -2$ lie outside the interval $[-2,2]$.
\end{enumerate}
\end{lemma}

In practice, these constants are often small and computable (indeed, often~$M$ is equal to one, as it
will be in our examples).
 We have, moreover, the following  easy upper bound for~$\lambda$:
 
\begin{lemma} Suppose that the largest valence of any vertex of~$\Gamma_{\k}$ 
is $v$. Then~$\lambda \le v$.
\end{lemma}

Combined with
Theorem~\ref{theorem:bounds} above, we deduce:

\begin{corr} For all sufficiently large~$\k$, the largest eigenvalue~$\lambda$ of
$\Gamma_{\k}$ satisfies~$\M(\lambda^2 - 2) < 14/5$.
\end{corr}

\begin{proof} Since~$\lambda^2 - 2$ is strictly increasing as~$\k$ increases, it suffices to show that
the degree of~$\lambda^2$ is not bounded. Yet all the conjugates of~$\lambda^2 - 2$ are bounded by~$L$,
and there are only a finite number of algebraic integers of fixed degree with this property, by 
a well known lemma of Kronecker.
\end{proof}

We shall  prove in Prop~\ref{prop:list} that if $\Q(\lambda^2)$ is abelian, 
then~$\M(\lambda^2 -2 ) < 14/5$ implies either that $\lambda \le 2$ or~$\lambda^2 - 2$ is 
one of a finite set of algebraic integers. This is is enough to prove that there are only finitely many abelian
 spiders which are not Dynkin diagrams for sufficiently large~$\k$. On the other hand, if one the~$r_i$ is
 bounded by a constant~$B$, then we can proceed by induction and consider the~$k-1$ spiders on the finitely many
 graphs where a $2$-valent tree of length $r_i \le  B$ is attached to~$\Gamma$ at~$v_i$. 
 This leads to a proof of Theorem~\ref{theorem:one}.
 The problem is that
Kronecker's theorem, although ``explicit,'' is not really so explicit in practice (since it involves
checking a super-exponential set of polynomials).
Instead, we shall give a different argument which can be used in practice.

\begin{prop} \label{prop:degree} Suppose that each element of~$\k$ is at least~$n \ge 2$. There is a bound:
$$D:=[\Q(\lambda^2):\Q] \gg n,$$
where the implied constant depends only on~$\Gamma$ and is explicitly computable.
\end{prop}

\begin{proof} We may assume that~$n$ is large (in practice, what counts as ``large'' is usually
not prohibitive). Write~$\lambda = \rho + \rho^{-1}$. Certainly $[\Q(\lambda^2):\Q(\rho)] \le 4$,
so it suffices to give a linear lower bound on the degree of~$\rho$. 
Let~$\rho_{\infty}$ denote the largest root of~$Q(t + t^{-1})$. We know that the values~$\rho$
are converging to~$\rho_{\infty}$;
the basic idea is to show that this convergence is exponentially fast, which,
together with the fact that the conjugates of~$\rho$ are constrained in absolute value, is enough to give the requisite
bound on the degree of~$\rho$.
Write~$\rho_{\infty} - \rho = \veps$.
Since~$P_{\k}(\rho + \rho^{-1}) = 0$, we deduce that:
$$0 =   P_{\k}\left(\rho + \rho^{-1} \right)= \sum_{\keps}
\rho^{\sum \eps_i r_i} F_{\keps}(\rho + \rho^{-1}).$$
Taking absolute values and applying the triangle inequality, we deduce that
$$
 | Q(\rho + \rho^{-1}) | \ll \rho^{-2n},$$
 where the constants can easily be made effective in any particular case (they involve
 the supremum of the polynomials~$F_{\keps}(t + t^{-1})$ for~$t$ in a neighbourhood of~$\rho_{\infty}$).
On the other hand, suppose that the root~$\rho_{\infty}$ of
$Q(x)$ has a multiplicity exactly~$m$.  Then 
there is an inequality
$|Q(\rho + \rho^{-1})| \gg   \veps^{m}$ for some explicitly computable
constant~$A > 0$ depending on the~$m$th derivative of~$Q$ at~$\rho_{\infty} + \rho^{-1}_{\infty}$.
Since~$\rho$ is converging to~$\rho_{\infty} > 1$, it satisfies~$\rho^{2/m} > \theta$ for some explicit~$\theta > 1$
which does not depend on~$n$.
It follows that, where as above the implicit constants can easily be evaluated
explicitly, we have the following inequality:
$$ |\rho_{\infty} - \rho | \ll  \frac{1}{\theta^{n}}.$$
Let $R(t)$ be the minimal polynomial of~$\rho_{\infty}$.
The polynomial~$R(t)$ does not vanish on any conjugate
of~$\rho$ because~$\rho_{\infty} > |\sigma \rho|$ 
for all conjugates of~$\rho$ and~$R(t)$ is irreducible.
The polynomial~$R(t)$ is bounded on the  ball $|t| \le \rho_{\infty}$ by some absolute constant~$C$.
Let~$D = [K:\Q]$ with~$K = \Q(\lambda^2)$, and let~$L = \Q(\rho)$.
Since~$[L:K] \le 4$,   the degree of~$L/\Q$ is at most~$4D$.
 Since $R(\rho) \ne 0$, we have
$$1 \le N_{L/\Q}(R(\rho)) \le
C^{4D-1} \cdot |\rho - \rho_{\infty}|
\le \frac{C^{4D-1}}{\theta^{n}}.$$
Taking logarithms and using the fact that~$\theta > 1$ leads to a linear lower bound in~$D$, as desired.
\end{proof}
		
Combining this result with 
Theorem~\ref{theorem:bounds} above, we deduce:

\begin{corr} \label{corr:deduce} There exists an effectively computable constant~$n$ such
that for all $r_i \ge n$,  either the largest eigenvalue~$\lambda$ of
$\Gamma_{\k}$ satisfies~$\M(\lambda^2 - 2) < 14/5$, or~$\Gamma_{\k}$ is an affine Dynkin
diagram.
\end{corr}

\begin{proof} The previous Lemma shows that we may find an explicit~$n$ so that the degree of~$\lambda^2$ is large.
The result then follows from~Theorem~\ref{theorem:bounds} and Lemma~\ref{lemma:uniform} once we have
an effective bound on~$n$ so that~$\lambda^2 - 2$ is not conjugate to a singularity of~$B(x)$. Note, however, that
all the singularities of~$B(x)$ are algebraic integers all of whose conjugates lie in $[0,4]$. If~$\lambda^2-2$ is such an integer,
then~$|\lambda| \le 2$ and~$\Gamma_{\k}$ is an (affine) Dynkin diagram.
\end{proof}

\section{Totally real cyclotomic integers with small \texorpdfstring{$\M$}{M}}

In this section,  we shall improve on some estimates from~\cite{CMS}. 
We make, however, the following preliminary remark.
Modifying the proof of Theorem~\ref{theorem:bounds} slightly, we see that
there exists a lower bound on~$D$ (depending on~$L$ and any~$\eps > 0$)
that guarantees the inequality $\M(\beta) \le 9/4 + \eps$. 
However,  this can be improved further.
The proof
of Lemma~\ref{lemma:sneak} has some slack can also be exploited, namely by replacing~$B(x)$
by~$B(x)$ by~$B(x) - \delta$ for small but non-zero~$\delta > 0$.
(As mentioned directly after the statement of Lemma~\ref{lemma:sneak},
one could take~$\delta$ to be anything less than approximately
$0.00599001$.)
This would allow us to modify the proof of~Theorem~\ref{theorem:bounds}
to give an explicit lower bound on~$D$ (in terms only of~$B(L)$) which would
guarantee that~$\M(\beta) < 9/4$.
We could then dispense
with~Prop.~\ref{prop:list} below entirely and use Lemma~9.0.1 of~\cite{CMS},
which classifies those~$\beta$ with~$\M(\beta) < 9/4$. However,
such an argument would lead to (significantly) worse bounds.

\medskip

We shall freely use many of the concepts from~\cite{Cassels} and~\cite{CMS}.
Recall that  two algebraic cyclotomic integers are called equivalent
if their ratio is a root of unity, and that a cyclotomic integer~$\beta$ is minimal
if it has smallest conductor amongst all its equivalent forms. If~$\beta$ is totally real,
it is not always the case that a minimal equivalent cyclotomic integer is also totally real,
but this is almost true:

\begin{lemma} \label{lemma:real} If~$\beta$ is a minimal  cyclotomic integer of odd conductor~$N$
which is equivalent to a cyclotomic integer, then, up to a root of unity in~$\Q(\zeta_{N})$,
 either~$\beta$ or~$\beta \cdot \sqrt{-1}$ is totally real.
\end{lemma}

\begin{proof}
Suppose that~$\beta$ is  minimal of conductor~$N$. Write $\gamma = \zeta \beta$, where~$\gamma$
is totally real.
If~$\zeta \in \Q(\zeta_{2N}) = \Q(\zeta_{N})$, then the result is trivial.
Hence there exists a prime~$p$ such that $2N$ is exactly divisible by $p^m$ and the order of~$\zeta$ is
exactly divisible by $p^{n}$ for some $n > m$.
Let~$\xi$ denote a primitive~$p^{n}$ root of unity.  
There exists a Galois automorphism~$\sigma$ fixing~$\Q(\zeta_{N})$ and hence fixing~$\beta$ such that
$$\gamma/\sigma \gamma = \sigma \zeta/\zeta = \xi^{p^{m}}.$$
Since~$\gamma$ is totally real, the latter element must also be real, which forces
$p = 2$ and $\xi^{4} = 1$ (noting that $p^m = 2$ if~$p = 2$, since~$N$ is odd). The result follows.
\end{proof}

\begin{remark} \emph{Let~$\alpha$ be a cyclotomic integer. Let~$\N(\alpha)$ denote
the minimum number of roots of unity required to express~$\alpha$. If~$\alpha \in K = \Q(\zeta_N)$,
let $\N_K(\alpha)$ denote the minimum number of roots of unity in~$K$ required to express~$\alpha$.
We recall the following facts from~\cite{Cassels,CMS} for cyclotomic
integers~$\alpha$:
\begin{enumerate}
\item If $\N(\alpha) > 1$ is not a root of unity, then $\M(\alpha) \ge 3/2$. (\cite{Cassels}, Lemma~2)
\item If $\N(\alpha) > 1$, and~$\alpha$ is not a root of unity times a conjugate of~$1 + \zeta_5$,
then~$\M(\alpha) \ge 5/3$.
\item If $\N(\alpha) \ge 3$, then~$\M(\alpha) \ge 2$.  (\cite{Cassels}, Lemma~3)
\end{enumerate}
}
\end{remark}

\begin{prop} \label{prop:list} Suppose that $\M(\beta) < 14/5$ and~$\beta$ is a totally
real cyclotomic integer. Suppose, moreover, that
$\beta$ is not the sum of at most two roots of unity. Then~$\ho{\beta}$ is one of the following  numbers:

\begin{center}
{\small
\begin{tabular}{rrr}
$\beta \qquad \qquad \qquad$ & $\M(\beta) $ & $[\Q(\beta):\Q]$ \\
\hline
$ \displaystyle{ \frac{\sqrt{3} + \sqrt{7}}{2}  = 2.188901\ldots}$ & $5/2$ & $4$ \\
 $  1 + 2 \cos(2 \pi/7)  = 2.246979\ldots$ &  $2$  & $3$ \\
 $  \zeta_{12} + \zeta_{20} + \zeta^{17}_{20}  = 2.404867 \ldots$ &  $2$  & $8$ \\
 $ 2 \cos(11\pi/42) + 2 \cos(13\pi/42) =  2.486985 \ldots$ &  $ 8/3  $  & $12$ \\
 $  1 + 2 \cos(2 \pi/11)  = 2.682507\ldots$ &  $12/5$ & $5$ \\
 $  1 + 2 \cos(2 \pi/13)  = 2.770912 \ldots$ &  $5/2$  & $6$ \\
 $  1 + 2 \cos(2 \pi/17)  = 2.864944 \ldots$ &  $21/8$ & $8$ \\
 $  1 + 2 \cos(2 \pi/19)  = 2.891634 \ldots$ &  $8/3$ & $9$  \\
 $  2\cos(2 \pi/35) + 2\cos(12 \pi/35) = 4 \cos(\pi/7) \cos(\pi/5) = 2.915596 \ldots$ &  $5/2   $ & $6$ \\
 $  1 + 2 \cos(2 \pi/23)  = 2.925834 \ldots$ &  $30/11$ & $11$ \\
 $  1 + 2 \cos(2 \pi/29)  = 2.953241 \ldots$ &  $39/14$ & $14$ \\
 $  1 + 2 \cos(2 \pi/30)  = 2.956295 \ldots$ &  $11/4$  & $4$ \\
 $  1 + 2 \cos(2 \pi/60)  = 2.989043 \ldots$ &  $11/4$  & $8$ \\
 $  \zeta_{84}^{-9} + \zeta_{84}^{-7} + \zeta_{84}^{3} + \zeta_{84}^{15} =  3.056668 \ldots$ &  $5/2   $ & $12$  \\
 $  2\cos(6\pi/55) + 2\cos(16\pi/55) = 4 \cos(\pi/11) \cos(\pi/5)  = 3.104984 \ldots$ &  $ 27/10  $ & $10$  \\
 $  2\cos(8\pi/65) + 2\cos(18\pi/65) = 4 \cos(\pi/13) \cos(\pi/5)  = 3.142033 \ldots$ &  $ 11/4  $ & $12$  \\
 $  2\cos(11 \pi/70) + 2\cos(17 \pi/70)  = 3.206780\ldots$ &  $8/3   $ & $24$ \\
 $  2\cos(37 \pi/210) + 2\cos(47 \pi/210)  = 3.227019 \ldots$ &  $ 11/4  $ & $24$ \\
 $  2\cos(\pi/42) + 2\cos(11 \pi/42)  = 3.354753 \ldots$ &  $8/3   $  & $12$ \\
 \end{tabular}
 }
 \end{center}
\end{prop}

\begin{proof} 
We may assume that $\N(\beta) \ge 3$.
Consider the case~$\N(\beta) = 3$.
By Theorem~4.0.3 of~\cite{CMS}, we may assume that, up to conjugation and sign, either
$\beta = 1 + \zeta + \zeta^{-i}$ for some root of unity~$\zeta$,  or
$\beta = \zeta_{12} + \zeta_{20} + \zeta^{17}_{20}$. 
The latter element is included on the list, the former elements
satsify~$\M(\beta) \le 14/5$ if and only if they are included in the statement of the theorem. Hence
we may assume that~$\N(\beta) \ge 4$.

Let us now weaken the assumption on~$\beta$ to assume merely that it is equivalent to a totally
real cyclotomic integer, and that~$\N(\beta) \ge 4$. This allows us to also assume that~$\beta$ is minimal, that is, it
lives in~$\Q(\zeta_N)$ where~$N$ is the conductor of~$\beta$, and no multiple of~$\beta$ by a root
of unity lives in a smaller cyclotomic field. 
Recall (following~\cite{Cassels,CMS}) that can write
$$\beta = \sum_{S} \alpha_i \zeta^i,$$
 where~$p^k \| N$,  where $\zeta$ is a primitive $p^{k}$th root of unity,
where $\alpha_i \in \Q(\zeta_M)$,  where $pM = N$, and where~$S$
is a subset of $\{0,1,\ldots,p-1\}$ of order which we  denote by~$X$.
Note that when~$p \| N$, this expression is only unique up to translating each~$\alpha_i$
by the same constant.

Assume that~$p^2 | N$ for some~$p$. Then $\M(\beta) = \sum \M(\alpha_i)$ (\cite{CMS}, Lemma~5.2.1).
If $|S| = X \ge 3$, then $\M(\beta) \ge 3$.
If $X =1$,  then $\beta = \alpha \zeta$, and we could divide by~$\zeta$, contradicting the minimality of~$\beta$.
If $X = 2$, then $\M(\beta) = \M(\alpha_1) + \M(\alpha_2)$.
The assumption $\N(\alpha_1) + \N(\alpha_2) > 3$ implies that
$\M(\beta) \ge 3/2 + 3/2 = 3$ or $\M(\beta) \ge 1 + 2 = 3$. This also contradicts our assumptions,  and so~$N$ is squarefree. Recall this
implies the equality (Eq.~3.9 of~\cite{Cassels}):
$$\M(\beta) = (p-X) \sum \M(\alpha_i) + \sum \M(\alpha_i - \alpha_j),$$
where we assume that exactly~$X$ of the~$\alpha_i$ are non-zero.

\medskip

Suppose that $p|N$ for some $p > 7$. Since $\M(\beta) < 7/2 \le (p+3)/4$,
then by Lemma~1 of~\cite{Cassels} (as used in~\cite{CMS}), we may assume that
there are exactly
of $X \le (p-1)/2$ non-zero terms~$\alpha_i$ in the expansion of~$\beta$ above. If $X \ge 4$, then
 we deduce that
$$(p-1) \M(\beta) \ge (p - X) X \ge 4(p-4).$$
This implies (for~$p > 7$) that~$\M(\beta)  \ge 14/5$.
 Suppose that~$X = 3$.
If $\alpha_i$ is a root of unity for each~$i$, then $\N(\beta) \le 3$, a contradiction. Hence at least 
one~$\alpha_i$ is not a root of unity. If all the~$\alpha_i$ are not roots of unity then~$(p-1) \M(\beta) \ge (p-3)(3/2)$
which directly leads to a contradiction. Otherwise, there must be at least two pairs which are non-zero, and so
$$(p-1) \M(\beta) \ge (p-3)(1 + 1 + 3/2) + 2,$$
from which $\M(\beta) \ge 3$.
Hence we may assume that~$X = 2$, and  in particular that
$$\beta = \alpha + \zeta \gamma,$$
where $\zeta$ is a primitive $p$th root of unity, $\alpha$ and $\gamma$ are cyclotomic integers in $\Q(\zeta_M)$ for
$M$ dividing~$N$ and prime to~$p$. 
Since~$\N(\beta) > 3$, either  $\alpha$ is a root of unity and $\N(\gamma) \ge 3$, or
$\alpha$ and~$\gamma$
are both not roots of unity. In the first case, $\N(\gamma - \alpha) \ge 2$ so~$\M(\gamma - \alpha)  \ge 2$.
Hence
$$(p-1) \M(\beta) \ge (p-2)(1 + 2) + 2,$$
and so~$\M(\beta) \ge 29/10$. 
In the second case, if~$\alpha \ne \gamma$, then
$$(p-1) \M(\beta) \ge (p-2)(3/2 + 3/2) + 1,$$
and~$\M(\beta) \ge 14/5$. If~$\alpha = \gamma$ and~$\M(\alpha) \ge 5/3$, then~$\M(\beta) \ge 3$.
So, after conjugation, we must have:
$$\beta = (1 + \zeta)(1 + \zeta_5).$$
In this case, we have~$\ho{\beta} = \ho{{1 + \zeta}} \cdot \ho{{1 + \zeta_5}}$. Note that~$\ho{{1 + \zeta}} = 2\cos(\pi/p)$. 
If~$p > 13$, then we have~$\M(\beta) \ge 45/32$, so this leaves only~$p = 11$ and~$p = 13$, and these
cases are covered in the statement of the theorem.
This portion of the argument  is the one which most strongly requires
the bound~$\M(\beta) < 14/5$ rather than~$\M(\beta) < 3$. 
In particular, all the integers
$4 \cos(\pi/p) \cos(\pi/5)$ for a prime~$p > 5$ will satisfy this bound.

\begin{lemma}  \label{lemma:exhaust}
If $\beta \in K = \Q(\zeta_{105})$ is a sum of~$4$ or~$5$ roots of unity in~$K$, and~$\beta$ is
equivalent to a totally real integer, then either~$\ho{\beta}$ is   one of the exceptions listed in the statement of the theorem, or $\M(\beta) \ge 14/5$.
\end{lemma}

\begin{proof} 
One proceeds by enumeration, after noting by Lemma~\ref{lemma:real} that~$\beta \in K$ is equivalent to 
a totally real integer if and only if $\beta$ times some $420$th root of unity is real.
\end{proof}

We let~$p = 5$, and
write $\beta = \sum \alpha_i \zeta^i$ where $\zeta^5 = 1$. We have the following
by Lemmas~7.0.1 and~7.0.3 of~\cite{CMS}:
\begin{lemma}  \label{lemma:useful} If~$\alpha \in L = \Q(\zeta_{21})$, 
\begin{enumerate}
\item If $\N_L(\alpha) \ge 2$, then $\M(\alpha) \ge 5/3$.
\item If~$\N_L(\alpha) \ge 3$, then~$\M(\alpha) \ge 2$.
\item If~$\N_L(\alpha) \ge 4$, then~$\M(\alpha) \ge 5/2$.
\item If~$\N_L(\alpha) \ge 5$, then~$\M(\alpha) \ge 23/6$.
\end{enumerate}
\end{lemma}
Since we are not assuming that~$N$
is divisible by~$5$, we have to allow the possibility that~$X = 1$.

We consider various cases:
\begin{enumerate}\item If $X = 1$, then $\beta \in \Q(\zeta_{21})$.
By the Lemma~\ref{lemma:useful}, we may assume that~$\N_L(\beta) > 5$.  Hence~$\M(\alpha) \ge 23/6$, which is a contradiction.
\item If $X = 2$, then we may write $\beta = \alpha + \gamma \zeta$ with $\alpha, \gamma \in \Q(\zeta_{21})$, and
we have the equality:
$$4 \M(\beta) = 3 \M(\alpha) + 3 \M(\gamma) + \M(\alpha - \gamma).$$
Since $\N_K(\beta) > 5$, we may assume that either $\N_L(\alpha), \N_L(\gamma) \ge 3$, or~$\N_L(\alpha) = 2$ and~$\N_L(\gamma) \ge 4$,
or~$\N_L(\alpha) = 1$ and~$\N_L(\gamma) \ge 5$. Using~Lemma~\ref{lemma:useful}, and the fact that
$\N_L(\alpha - \gamma) \ge \N_L(\gamma) - \N_L(\alpha)$, we find in each case that:
$$\M(\beta) \ge \frac{1}{4} (3 \cdot 2 + 3 \cdot 2) = 3,$$
$$\M(\beta) \ge \frac{1}{4} (3 \cdot 5/3 + 3 \cdot 5/2 + 5/3) = 85/24,$$
$$\M(\beta) \ge  \frac{1}{4} (3 + 3 \cdot 5/2 + 5/2) = 13/4,$$
which all yield
 contradictions.
\item If $X = 3$, then, as in  the proof of the similar step in Lemma~9.0.1 of~\cite{CMS}, not all the~$\alpha_i$ can be the same (since otherwise we could reduce to the case~$X = 2$), and hence
at least two of the $\alpha_i - \alpha_j$ are non-zero. More generally, we have
$$4 \M(\beta) \ge 2 \sum \M(\alpha_i) + \sum \M(\alpha_i - \alpha_j).$$
The values~$(1,1,1)$, $(1,1,2)$, $(1,2,2)$, $(1,1,3)$, are ruled out as values of~$\{\N_L(\alpha_i)\}$ by Lemma~\ref{lemma:exhaust}.
This leaves the possibilities:
$$(1,1,>3), (1,2,>2), (1,>2,>2), (>1,>1,>1).$$
Considering each in turn and using Lemma~\ref{lemma:useful}, along with the fact that not all the~$\alpha_i$ are equal
in the final case, we have the four estimates:
$$\M(\beta) \ge \frac{1}{4} (2 \cdot 1 + 2 \cdot 1 + 2 \cdot 5/2 + 2 + 2) = 13/4,$$
$$\M(\beta) \ge \frac{1}{4} (2 \cdot 1 + 2 \cdot 5/3 + 2 \cdot 2 + 1 + 1 + 5/3) = 13/4,$$
$$\M(\beta) \ge \frac{1}{4} (2 \cdot 1 + 2 \cdot 2 + 2 \cdot 2 + 5/3 + 5/3) = 10/3,$$
$$\M(\beta) \ge \frac{1}{4} (2 \cdot 5/3 + 2 \cdot 5/3 + 2 \cdot 5/3 + 1 + 1) = 3,$$
which all lead to a contradiction.
\item If $X = 4$ or $X = 5$, we may reduce to $X \le 3$ exactly as in the proof of Lemma~9.0.1 of~\cite{CMS}.
\end{enumerate}
\end{proof}

\section{Proof of Theorem~\ref{theorem:one}}

Recall that the Perron--Frobenius eigenvalue~$\lambda$ of a graph~$\Gamma$
is $< 2$ (respectively, $\le 2$) if and only~$\Gamma$ is a Dynkin diagram (respectively,
affine Dynkin diagram). For topological reasons, only finitely many of the spiders~$\Gamma_{\k}$
are affine Dynkin diagrams.

Assume that infinitely many of the~$\Gamma_{\k}$ are abelian. We proceed
 by induction on~$k$, the result for~$k = 0$ being trivial.
If there exist infinitely many such graphs with $r_1 \le M$, then we may reduce the problem
to~$k - 1$ replacing~$\Gamma$ by the finitely many $1$-spiders on~$\Gamma$ with a $2$-valent tree
of length $\le M$ attached to~$\Gamma$ at~$v_1$. Hence we may assume that all the $r_i$ are tending to infinity.
If the limit of the~$\lambda$ as~$\k$ increases is $\le 2$, then all the~$\Gamma_{\k}$ are Dynkin diagrams.
Hence we may assume that the limit of the largest eigenvalue~$\lambda$ is $> 2$. 
By Prop.~\ref{prop:degree}, we obtain a lower bound on~$[\Q(\lambda)^2:\Q]$ which allows us (for sufficiently
large~$n$) to deduce as in Corr.~\ref{corr:deduce} that~$\M(\lambda^2 -2 ) < 14/5$. Since~$\lambda > 2$,
it is not the sum of two roots of unity. It follows
that~$\lambda> 2$ must be one of the finitely many exceptional numbers occurring in Prop.~\ref{prop:list}.
Yet these numbers have (explicitly) bounded degree, and so using the lower
bounds on~$[\Q(\lambda^2):\Q]$ in the proof of Prop.\ref{prop:degree}, these eigenvalues
can occur as~$\lambda^2 - 2$ for only finitely many~$\Gamma_{\k}$. 
 Hence we may explicitly compute such an~$n$ such that~$\Gamma_{\k}$ is not abelian when
 each~$r_i \ge n$,
which completes the proof of Theorem~\ref{theorem:one}.

\section{Examples}

We shall consider two examples. Let~$\Gamma_{a,b}$ be the Morrison spider given in the introduction.
Let $P_{a,b}(x)$ denote the characteristic polynomial of~$\Gamma_{a,b}$. We find:

\begin{lemma} There is an equality
$$ \left(t - \frac{1}{t}\right)^2 P_{a,b} \left(t + \frac{1}{t} \right)
= F_{a,b}(t) + F_{a,b}(1/t),$$
where
$$F_{a,b}(t) = t^{a+b}
(t^{-2} + 2 + 2 t^2 - 2 t^4 - 2 t^6 - 2 t^8 + t^{10}) + t^{a - b} (t^{-6} - 2 + t^6) .$$
\end{lemma}

Let~${\rho_{\infty}} =  1.6826\ldots $ be the largest real root of $t^6 - 2 t^4 - 2 t^2 - 1 = 0$, which is
 also a root of
$$t^{-2} + 2 + 2 t^2 - 2 t^4 - 2 t^6 - 2 t^8 + t^{10} = 0$$
(the other roots of this polynomial are cyclotomic).
Let~$\gamma = \left({\rho_{\infty}} + {\rho^{-1}_{\infty}} \right)^2 = 5.18438\ldots$ denote  the largest real root of
$$x^3 - 6 x^2 + 5 x - 4 = 0.$$
The following is the specialization of Lemma~\ref{lemma:general}:

\begin{lemma}
The polynomial $P_{a,b}(x)$ has a unique pair of roots $(\lambda,-\lambda)$ of
absolute value $> 2$. As~$a$ and~$b$ strictly increase, the value of $\lambda$ strictly increases.
The limit as $a,b \rightarrow \infty$ of~$\lambda^2$ is~$\gamma$.
\end{lemma}

We now find an explicit exponential bound relating~$\lambda$ to~$\gamma$.

\begin{lemma} \label{lemma:estimate}
Let $\rho \in [3/2,{\rho_{\infty}})$ denote the largest root of $P_{a,b}(t+t^{-1})$, and assume $a,b \ge n \ge 10$.  Then
$$|\rho - {\rho_{\infty}}| < \frac{1}{6} (1.682)^{-2n}.$$
\end{lemma}

\begin{remark} \emph{When we write a real number as a finite decimal, we refer to an exact
element of~$\Z[1/10]$. Although the inequalities below are quite tight, they still hold by some
comfortable margin of error. Certain numbers are chosen to make various ratios integral, purely
for presentation purposes.}
\end{remark}

\begin{proof}  Write $\rho = {\rho_{\infty}} - \veps$.
For~$a,b \ge n \ge 10$, we have the estimate $\rho  \in [1.682,{\rho_{\infty}})$.
In this range, the following inequalities hold:
$$|\rho^{-2} + 2 + 2 \rho^2 - 2 \rho^4 - 2 \rho^6 - 2 \rho^8 + \rho^{10}|  > 270 \cdot \veps,$$
$$|\rho^2 + 2 + 2 \rho^{-2} - 2 \rho^{-4} - 2 \rho^{-6} - 2 \rho^{-8} + \rho^{-10}| < 6,$$
$$|\rho^6  - 2 + \rho^{-6}| < 21.$$
The first inequality is obtained by looking at the derivative of this rational function
in the interval~$[1.682,\rho_{\infty}]$; the other inequalities are easy.
Using the equality $P_{a,b}(\rho + \rho^{-1}) = 0$ together with the triangle inequality, we find that
$$\begin{aligned}
 270 \cdot  \veps \cdot \rho^{2n} \le & \  |\rho^{-2} + 2 + 2 \rho^2 - 2 \rho^4 - 2 \rho^6 - 2 \rho^8 + \rho^{10}| 
 \cdot \rho^{2n}  \\
 \le &  \ |\rho^6  - 2 + \rho^{-6}| +  |\rho^2 + 2 + 2 \rho^{-2} - 2 \rho^{-4} - 2 \rho^{-6} - 2 \rho^{-8} + \rho^{-10}|
 \cdot
 \rho^{-2n} \\
 \le &  \ 
 42 + 6 \cdot  \rho^{-2n} \\
 \le & \ 45. \end{aligned}$$
The result follows.
\end{proof}

\begin{lemma} \label{lemma:estimatetwo}
If $a,b \ge n \ge 10$, and~$\lambda$ is the Perron Frobenius eigenvalue
of~$\Gamma_{a,b}$, then
$$ |\lambda^6 - 6 \lambda^4 + 5 \lambda^2 - 4| \
\cdot |\lambda^2|^{29/1000} \cdot |\lambda^2 - 2|^{14/100}
\cdot |\lambda^2 - 3|^{471/1000} \cdot |\lambda^2 - 4|^{362/1000}
\cdot |\lambda^6 - 6 \lambda^4 + 9 \lambda^2  - 1|^{8/625}$$
is bounded above by $23 \cdot (1.682)^{-2n}$.
\end{lemma}

\begin{proof} The function is decreasing on the interval~$[1.618,\rho_{\infty}]$. Hence,
by interlacing, it suffices to consider the case $a=b=n$.
 The result is then  an elementary calculus exercise from Lemma~\ref{lemma:estimate}.
 The main point is that if one replaces~$\lambda$ in the above expression by~$t + 1/t$, the
 resulting expression has derivative
 $< 138 = 6 \times 23$ in $[1.618,\rho_{\infty}]$  (for comparison, the exact value at~$\rho_{\infty}$ is
  approximately $\sim 136.12$).
\end{proof}

We can now give a lower bound on the degree of~$\lambda^2$, following
the argument of Prop~\ref{prop:degree}. 

\begin{prop} \label{prop:degreetwo} Suppose that $a,b \ge n \ge 10$. Then
$$D = [\Q(\lambda^2):\Q] > \frac{11}{25} \cdot n - \frac{1}{3}.$$
\end{prop}

\begin{proof} All proper conjugates $\sigma \lambda^2 \ne \lambda^2$
satisfy $0 < \sigma \lambda^2  < 4$. Hence (by calculus) 
$$ |\lambda^6 - 6 \lambda^4 + 5 \lambda^2 - 4| \
\cdot |\lambda^2|^{29/1000} \cdot |\lambda^2 - 2|^{14/100}
\cdot |\lambda^2 - 3|^{471/1000} \cdot |\lambda^2 - 4|^{362/1000}
\cdot |\lambda^6 - 6 \lambda^4 + 9 \lambda^2  - 1|^{8/625}$$
is bounded above in this interval by $10.56$.
(In contrast to the proof of Prop~\ref{prop:degree}, we include here some extra factors of $\lambda^2 - m$ for small~$m$ to
mollify the first factor as much as possible.)
On the other hand, since $\lambda \ne {\rho_{\infty}}$ is an algebraic integer, if $K = \mathbf{Q}(\lambda^2)$,
the product of the expression above over all conjugates of~$\lambda$ (assuming it is non-zero)
is a product of positive rational powers of norms, and is thus $\ge 1$. Using
the inequality above for $\sigma \lambda^2 \ne \lambda^2$ and Lemma~\ref{lemma:estimatetwo}
for $\sigma \lambda^2 = \lambda^2$, it follows that
$$1 
< 23 \cdot (1.682)^{-2n} \cdot (10.56)^{D- 1}.$$
If the degree $D$ is less than the value in the theorem, the RHS is less than one. 
\end{proof}
We deduce:

\begin{prop}
Suppose that~$a,b \ge 56$. Then~$\Gamma_{a,b}$ is not abelian.
\end{prop}

\begin{proof} In the context of Theorem~\ref{theorem:bounds}, with
$\beta = \lambda^2 - 2$ we have~$M = 1$ and~$L = \gamma - 2$, where
$B((\gamma - 2)^2) \sim -13.1241\ldots$. 
This yields the upper bound~$\M(\lambda^2 - 2) < 14/5$ as soon
as $D \ge 20|B|/11$, or when~$D \ge 24$. By Prop~\ref{prop:degreetwo}, we have~$D > 24$ as soon
as~$n \ge 56$. Hence, in this range,~$\lambda^2 -2$ must be one of the exceptions listed
in Prop.~\ref{prop:list}. On the other hand, for~$n$ in this range, we also have the estimate
$3.17438\ldots < \gamma -2 - 1/100 < \lambda^2 - 2 < \gamma - 2 = 3.18438 \ldots$, which certainly rules out all such 
exceptions.
\end{proof}

To complete the proof of Theorem~\ref{theorem:morrison}, it suffices to consider the case when~$a \le 56$
or~$b \le 56$ (since the polynomial~$P_{a,b}(x)$ is symmetric in~$a$ and~$b$, we may assume the former).
However, we can now apply the algorithm of~\cite{CMS} to rule out the remaining cases (we thank Scott Morrison
for carrying out this computation). We could also
rule out the cases using the methods in this paper, however, we omit the details for reasons of space, and
because we include the relevant details in the case of~$3$-spiders below.

\section{\texorpdfstring{$3$}{3}-Spiders}

We consider the case when~$k = 3$ and~$\Gamma$ is a single point. Let the resulting~$3$-spider be denoted~$\Gamma_{a,b,c}$.

\begin{theorem} \label{theorem:3spider} The complete set of abelian~$3$-spiders is as follows:
\begin{enumerate}
\item Those that are Dynkin diagrams, equivalently, those with $\lambda^2 \le 4$:
\begin{figure}[H] \label{fig:two}
\begin{center}
  \includegraphics[width=60mm]{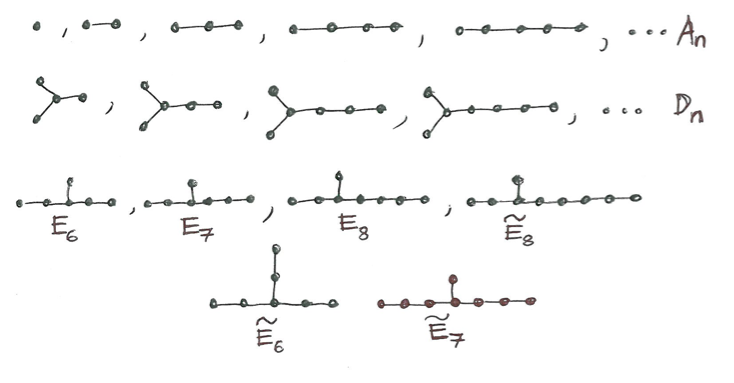}
  \end{center}
\end{figure}
\item Exactly three spiders with $\displaystyle{\lambda^2 = \frac{5 + \sqrt{13}}{2} = 4.302775\ldots}$
\begin{figure}[H] \label{fig:three}
\begin{center}
  \includegraphics[width=60mm]{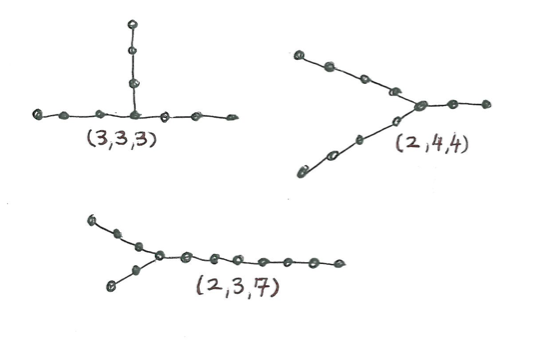}
  \end{center}
\end{figure}
\item Exactly three spiders with $\displaystyle{\lambda^2 = \zeta^{11} + \zeta^{10} + \zeta^{3} + \zeta^2 + 2 = 4.377202 \ldots}$,
where $\zeta = \exp(2 \pi i/13)$:
\begin{figure}[H] \label{fig:four}
\begin{center}
  \includegraphics[width=60mm]{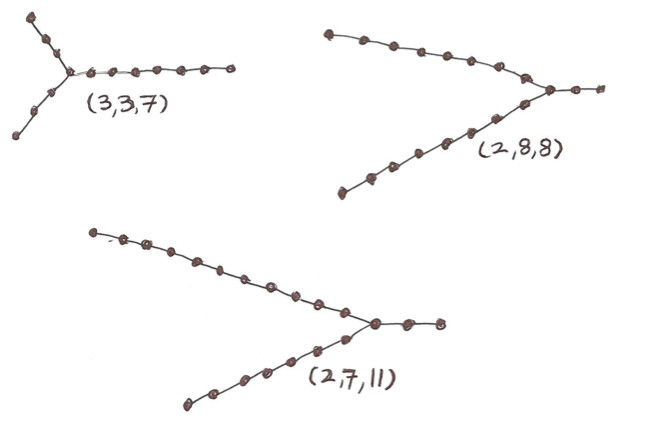}
  \end{center}
\end{figure}
\item Exactly three spiders with $\displaystyle{\lambda^2 =3 + \sqrt{2} =  4.414213 \ldots}$, namely:
\begin{figure}[H] \label{fig:five}
\begin{center}
  \includegraphics[width=60mm]{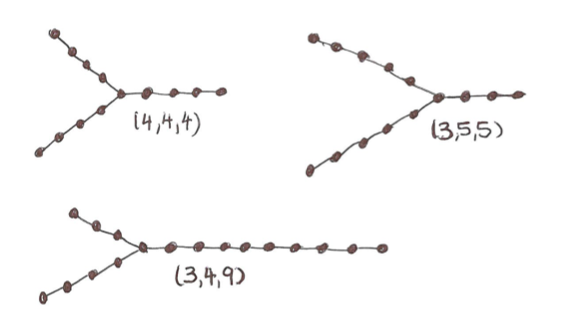}
  \end{center}
\end{figure}
\end{enumerate}
\end{theorem}

\begin{remark} \emph{The first two non-Dynkin diagrams include graphs
which correspond  to the Haagerup and extended Haagerup subfactors respectively 
(namely, the~$(3,3,3)$ and~$(3,3,7)$ spiders). However, none of the final class of graphs
correspond to any subfactors, because the index is~$< 5$ and does not occur as an index
of any possible subfactor in the classification~\cite{class}.}
\end{remark}

By symmetry, we may assume that $a \le b \le c$.
Let $P_{a,b,c}(x)$ denote the characteristic polynomial of~$\Gamma_{a,b,c}$. Using Lemma~11 of~\cite{Salem}, one easily establishes the following equality:

\begin{lemma} There is an equality
$$\begin{aligned} P(t) \left(t - t^{-1}\right)^3 (-1)^{a+b+c-1} = & \  t^{a+b+c+4} - 2 t^{a+b+c+2}  + t^{a+b-c} + t^{a+c-b} + t^{b+c-a} \\
& \  - t^{a-b-c} - t^{b-a-c} - t^{c-a-b}
+ 2 t^{-a-b-c-2} - t^{-a-b-c-4}. \end{aligned}$$
\end{lemma}

It is easy to identify the triples~$(a,b,c)$ such that~$\Gamma_{a,b,c}$ is a Dynkin diagram,
so we assume that the Perron--Frobenius eigenvalue~$\lambda$ of~$\Gamma$ is always strictly
larger than~$2$.
From Lemma~\ref{lemma:general}, we deduce that
the polynomial $P_{a,b,c}(x)$ has a unique pair of roots $(\lambda,-\lambda)$ of
absolute value $> 2$, and that the
 limit as $a,b,c \rightarrow \infty$ of~$\lambda$ is~$3/\sqrt{2}$.

\begin{prop} Let $D = [\Q(\lambda^2):\Q]$. Then  either
$\M(\lambda^2 - 2) < 14/5$ or~$D \le 12$.
\end{prop}

\begin{proof} It suffices to note that (taking~$\beta = \lambda^2 - 2$) that at most
one conjugate of~$\beta$ lies outside~$[-2,2]$, and so we 
deduce the inequality on~$\M(\lambda^2 - 2)$
from
Theorem~\ref{theorem:bounds} providing
$$D  \ge 2 \cdot \left|  B \left(\frac{25}{4}\right) \right| = 12.904524\ldots$$
\end{proof}

Let us now make the running assumption that~$\M(\lambda^2 -2) \ge 14/5$; we shall
deal with the alternative below. It follows that we may assume that
$D = [\Q(\lambda^2):\Q]$ is at most~$12$. 

\begin{lemma} With $a \le b \le c$, we have~$a \le 30$.
\end{lemma}

\begin{proof} Since it is useful to have a tight a bound as possible, instead of  using the
trivial bound~$|2x-9| \le 9$ on $[0,4]$ we note that
$$|2x-9| \cdot |x|^{52/100} \cdot |x - 1|^{337/1000} \cdot |x-2|^{3/10} \cdot |x-3|^{13/100}  < 5.58,$$
for~$x \in [0,4]$, a fact which is tedious but elementary to prove by calculus.
By giving a lower estimate for the derivative of this function in a neighbourhood of $9/2$, we also find that
$$|2\lambda^2-9| \cdot |\lambda^2|^{52/100} \cdot |\lambda^2 - 1|^{337/1000} \cdot |\lambda^2-2|^{3/10} \cdot |\lambda^2-3|^{13/100}  < 4.63 \cdot |(2 \lambda^2 - 9)|,$$
for all~$\lambda$.
Taken together, we deduce that
$$1 \le (5.58)^{D-1} \cdot 4.63 \cdot 
 |(2 \lambda^2 - 9)|,$$
 and hence, since~$D \le 12$,
 $$\left| \lambda^2 - \frac{9}{2} \right| > 6.6132 \ldots \times 10^{-10}.$$
 This inequality is violated as soon as~$a > 30$.
\end{proof}

\subsection{Fixed~\texorpdfstring{$a$}{a}, and 
varying~\texorpdfstring{$b$}{b} and~\texorpdfstring{$c$}{c}}
In this section, we effectively consider the~$2$-spiders on the Dynkin diagram~$\Gamma = A_{n}$ 
with~$n = a+1$ and $v_1 = v_2$ a terminal point of~$\Gamma$. Hence, for this section,
the values of~$\rho_{\infty}$ reflects the appropriate root of the new polynomial~$Q(t+t^{-1})$
in this setting. Note, however, that we still know that~$\lambda^2$ has a unique conjugate
outside the range~$[0,4]$. 

Suppose that~$a$ is fixed, and  let~$c$ and~$b$ with $c \ge b  \ge a$ vary without bound.
If one writes~$\lambda^2 - 2 = \rho^2 + \rho^{-2}$, then~$\rho^2$ is a Salem number, that is,
all the conjugates of~$\rho^2$ beside~$\rho^{-2}$ have modulus one. Since we
are assuming $D=[\Q(\lambda^2):\Q] \le 12$, we also have the inequality
 $[\Q(\rho^2):\Q] \le 24$. As~$b$ and $c$ tend to infinity,~$\rho$ tends towards the (unique)
 largest root~$\rho_{\infty}$ of the polynomial~$1 - 2t^{2a+2} + t^{2a+4}$, which
 is the polynomial~$Q(t + t^{-1})$ (up to powers of~$t^{\pm 1}$) of~\S~\ref{section:spectrum}

\begin{lemma} \label{lemma:foursix}
We have an inequality:
$$|1 - 2 \rho^{2a +2} + \rho^{2a + 4}| > \frac{1}{4^{23}}.$$
\end{lemma}

\begin{proof}  Since~$\rho_{\infty} > \rho$ is the only real root of this polynomial greater than one,
it follows that neither~$\rho$ nor any of its conjugates is a root of this polynomial.
For any non-trivial conjugate of~$\rho^2$, we have
the easy estimate $|1 - 2 \sigma \rho^{2a+2} + \sigma \rho^{2a + 4}| \le 4$, with a strict inequality
for the real root. Hence the result follows from the fact that the norm of $1 - 2 \rho^{2a+2} + \rho^{2a+4}$
from~$\Q(\rho^2)$ to~$\Q$
has absolute value at least one, and that the degree of~$\rho^2$ is at most~$24$.
\end{proof}

By interlacing, the root~$\rho$ increases with~$b$ and~$c$. Hence, by checking
for suitable choices of~$b$ and~$c$, we immediately deduce:

\begin{lemma} For each~$a$, we have the following upper bound on~$b = \min(b,c)$:
\begin{center}
\begin{tabular}{cc|cc|cc}
$a$ & $\min(b,c)$ & $a$ & $\min(b,c)$ & $a$ & $\min(b,c)$  \\
\hline
$1$ & $67$ & $11$ & $59$		& $21$ & $69$ \\
$2$ & $55$ & $12$ & $60$	& $22$ 		& $70$ \\
$3$ & $53$ & $13$ &  $61$		& $23$  & $71$ \\
$4$ & $53$ & $14$	& $62$		& $24$   & $72$ \\
$5$ & $53$ & $15$	& $63$		& $25$   & $73$ \\
$6$ & $54$  & $16$	& $64$		& $26$  & $74$ \\
$7$ & $55$ & $17$	& $65$		& $27$  & $75$ \\
$8$ & $56$ & $18$	&  $66$		& $28$  & $76$ \\
$9$ & $57$ & $19$  &  $67$		& $29$  & $77$ \\
$10$ & $58$ & $20$ &  $68$		& $30$  & $78$ \\
\end{tabular}
\end{center}
\end{lemma}

\subsection{Fixed~\texorpdfstring{$a$}{a} and 
~\texorpdfstring{$b$}{b}, and varying~\texorpdfstring{$c$}{c}}
	
We have reduced to a finite number of pairs~$(a,b)$, and we could finish
with an appeal to~\cite{CMS}. Instead, however, we give a 
 a treatment similar to the case when~$a$ is fixed and~$b$ and~$c$ are varying.
 As in the previous section, we assume~$c \ge b \ge a$, and redefine the polynomials~$Q(t+t^{-1})$
 and~$\rho_{\infty}$ (for each~$(a,b)$) to be the corresponding values for these~$1$-spiders.
 
 \begin{lemma}
We have an inequality:
$$| {\rho}^{2a+2b+4} - 2  {\rho}^{2a+2b+2} +  {\rho}^{2b} +  {\rho}^{2a} - 1| > \frac{1}{6^{23}}.$$
\end{lemma}

\begin{proof} The proof is the same as the proof of Lemma~\ref{lemma:foursix}; the polynomial
above is the minimal polynomial of~$\rho_{\infty}$. 
\end{proof}

By interlacing and computing the values of~$\rho$ for various triples~$(a,b,c)$,
we deduce:

\begin{lemma} If~$\Gamma_{a,b,c}$ is abelian, then one of the following holds:
\begin{enumerate}
\item There are bounds~$a \le 30$, $b \le 78$,
and~$c \le 170$.
\item $\M(\lambda^2 - 2) < 14/5$.
\end{enumerate}
\end{lemma}

We now complete the proof of Theorem~\ref{theorem:3spider}. Suppose that~$\M(\lambda^2 -2) < 14/5$. 
Then by Prop.~\ref{prop:list}, we deduce that either~$\lambda^2 - 2$ is a sum of two roots of unity
or less (from which it follows immediately that~$\Gamma_{a,b,c}$ is an (affine) Dynkin diagram, or
$\lambda^2 -2$ is one of the following numbers:
$$\alpha =  \frac{\sqrt{3} + \sqrt{7}}{2},  \beta = 1 + 2 \cos(2 \pi/7),  \gamma =
 \zeta_{12} + \zeta_{20} + \zeta^{17}_{20},
\delta = 2 \cos(11\pi/42) + 2 \cos(13\pi/42),$$
where we use the fact that~$\beta^2 < 9/2$. The algebraic
numbers~$\alpha$ and~$\delta$ have conjugates $< 2$, 
 yet~$\lambda$ is totally real, so~$\lambda^2 - 2$ has no such conjugate.
In the second and third cases, we have
$$\beta \sim 2.060820\ldots  \ \text{or} \ \gamma \sim 2.098777\ldots $$
We dispense with these possibilities (for $(a,b,c)$ outside the bounds in part one) by the following argument:
\begin{enumerate}
\item 
If $a \ge 3$, then $\lambda > 2.074313\ldots >  2.060820\ldots$
\item If $a = 2$ and $b \ge 4$, then $\lambda > 2.074313\ldots > 2.060820\ldots$
\item If $a = 2$ and $b = 2$, then $\lambda < \sqrt{2+\sqrt{5}} < 2.060820\ldots$
\item If $a = 2$, $b = 3$,  and $c \ge 5$, then $\lambda > 2.069782\ldots >  2.060820\ldots$
\item If $a =1$, then $\lambda < \sqrt{2+\sqrt{5}} < 2.060820\ldots$
\end{enumerate}
for the first case, and
\begin{enumerate}
\item 
If $a \ge 4$, then $\lambda > 2.101002\ldots > 2.098777\ldots$
\item If $a = 2$ and $b \ge 5$, then $\lambda >  2.101002\ldots > 2.098777\ldots$
\item If $a = 2$ and $b \le 4$, then $\lambda < 2.084868\ldots < 2.060820\ldots$
\item If $a \le 2$,  then $\lambda < 2.093555\ldots  < 2.098777\ldots$
\end{enumerate}
in the second.
Finally, we check all the remaining polynomials to see which give rise to abelian extensions.
We say a few words about this computation. The first step consists of looping through the polynomials
(which have root~$\rho$)
and dividing through by the cyclotomic factors. If the remaining polynomial
is irreducible and of degree $\ge 48$, then we are done. 
Degree considerations  eliminated all polynomials with~$a \ge 12$ except some of the form
$(a,b,c)= (a,a+1,2a+3)$, $(a,a+2,a+2)$ or $(a,a,a)$. The polynomial was irreducible
except for a few exceptional cases, namely, $(a,b,c) = (2,6,20)$, and the triple
of graphs $(4,8,14)$, $(4,9,9)$,  and $(5,5,8)$. The latter triple is somewhat interesting --- the value of~$\lambda^2 -2$
in each case is the largest real root of~$\theta^3 - 2 \theta^2 - 4 \theta + 7 =0$, whose splitting field is the Hilbert class field
of~$\Q(\sqrt{229})$. 
The second check consisted of computing the corresponding minimal polynomial of~$\lambda^2 - 2$,
and then checking  (using \texttt{polcompositum} in \text{gp/pari}) whether the field was Galois or not.  Finally, it was checked whether any of the fields thus obtained were abelian or not (there were no false positives).

\subsection{Miscellaneous Applications}

All Salem numbers~$\rho$  are reciprocal. If~$\Q(\rho)$ is abelian, then since~$\rho$ is real,
it must be totally real, yet~$\rho$ (by definition) has a root of absolute value~$1$. Thus no
Salem number can generate an abelian extension. In light of this, the following definition
is perhaps not too confusing.

\begin{df} A Salem number~$\rho$ is of abelian type if~$\Q(\rho + \rho^{-1})$ is an abelian extension.
\end{df}

If~$K$ is any totally real field, then, because the image of the units~$\OL^{\times}_K \otimes \R$ in
$K \otimes \R$ has co-dimension one (by the proof of Dirichlet's unit theorem), there exists a totally
positive unit~$\alpha \in \OL_K$ such that $\alpha > 1$ in one complex embedding and $< 1$ in all
other complex embeddings. Replacing~$\alpha$ by a suitable power so that it is $>2$ and letting
~$\rho + \rho^{-1} = \alpha$, we find that~$\rho$ is a Salem number of abelian type if~$K$ is a (totally) real abelian
field. Hence there exist an abundance of Salem numbers of abelian type. However, we prove the following:

\begin{prop} The set of Salem numbers of abelian type is discrete in~$\R$.
\end{prop}

\begin{proof}
It suffices to show that the  Salem numbers of abelian type less than a given bound~$L$ is finite.
Since the number of Salem numbers of bounded degree less than a given bound is finite, it
suffices to prove that the Salem numbers of abelian type less than~$L$ have bounded degree.
However, from Theorem~\ref{theorem:bounds}, for all Salem numbers of sufficiently large degree
(depending on~$L$) we have the bound~$\M(\rho + \rho^{-1}) < 14/5$. 
If~$\rho$ is of abelian type,
then by Prop.~\ref{prop:list}, the element~$\rho$ lives in some finite set (if~$\rho$ is Salem,
then~$\rho + \rho^{-1} > 2$ is not a sum of two roots of unity).
 \end{proof}

Note that, from the classification of the smallest totally real cyclotomic integers~\cite{CMS}, one sees that the smallest Salem number of abelian type is~$\theta = 1.635573\ldots$, the root of
$\theta^6 - 2 \theta^5 + 2 \theta^4 - 3 \theta^3 + 2 \theta^2 - 2 \theta + 1 = 0$. 

One can make the previous proposition effective. Namely, suppose that $\rho > \rho'$
are two Salem numbers of abelian type. There is a bound~$B(x^2 + x^{-2} + 2) > - 11/10 x^2$ for all~$x > \theta$. Hence
$$\frac{20}{11} \cdot B(\rho^2 + \rho^{-2} + 2) \ge - 2 \rho^2,$$
and so either the degrees of~$\rho$ and~$\rho'$ are either bounded by $4 \rho^2$, or
the corresponding Salem numbers lie on the list in Prop.~\ref{prop:list}, in which case one can
check that the bound still holds.
In the former case, by estimating the norm of~$\rho - \rho'$, which has degree at most~$16 \rho^4$ and each conjugate has absolute value at most~$2 \rho$, we deduce that

\begin{prop} Let $\rho > \rho'$ be two Salem numbers of abelian type. Then
$$\rho - \rho' > \frac{1}{(2 \rho)^{16 \rho^4}}.$$
\end{prop}

Naturally enough, the same result (and proof) hold if one replaces Salem numbers by numbers~$\rho$
conjugate to~$\rho^{-1}$ 
with a uniformly bounded number of real roots $> 1$.

\bibliographystyle{amsalpha}
\bibliography{Spider}

\end{document}